\documentclass[12pt,reqno]{amsart}
\usepackage{amsmath}
\usepackage{amssymb,amscd}

\numberwithin{equation}{section}

\newtheorem{theorem}{Theorem}[section]
\newtheorem{corollary}[theorem]{Corollary}
\newtheorem{lemma}[theorem]{Lemma}

\newtheorem{question}[theorem]{Question}

\theoremstyle{definition}
\newtheorem{defn}[theorem]{Definition}

\def \mc{\mathcal}

\def \({\left(}
\def \){\right)}
\def \<{\langle}
\def \>{\rangle}

\begin{document}

\title[On equivariant Serre problem for principal bundles]{On equivariant Serre problem for 
principal bundles }

\author[I. Biswas]{Indranil Biswas}
\address{School of Mathematics, Tata Institute of Fundamental Research, Mumbai, India }

\email{indranil@math.tifr.res.in}

\author[A. Dey]{Arijit Dey}

\address{Department of Mathematics, Indian Institute of Technology-Madras, Chennai, India}

\email{arijitdey@gmail.com}

\author[M. Poddar]{Mainak Poddar}

\address{Middle East Technical University, Northern Cyprus Campus, Guzelyurt, Mersin 
10, Turkey}

\email{mainakp@gmail.com}



\begin{abstract} 
Let $E_G$ be a $\Gamma$--equivariant algebraic principal $G$--bundle over a
normal complex affine 
variety $X$ equipped with an action of $\Gamma$, where 
$G$ and $\Gamma$ are complex linear algebraic groups. Suppose $X$ is contractible as a 
topological $\Gamma$--space with a dense orbit, and $x_0 \in X$ is a $\Gamma$--fixed point. We show that if
$\Gamma$ is reductive, then $E_G$ admits a $\Gamma$--equivariant isomorphism with the product
principal $G$--bundle $X \times_{\rho} E_G(x_0)$, where $\rho\,:\, \Gamma \,
\longrightarrow\, G$ is a homomorphism between algebraic groups. As a consequence, any torus equivariant 
principal $G$-bundle over an affine toric variety is equivariantly trivial. This leads to a classification of torus 
equivariant principal $G$-bundles over any complex toric variety, generalizing the main result of \cite{BDP}. 
\end{abstract}

\maketitle

\section{Introduction}

Let $\Gamma$ be a reductive complex algebraic group acting algebraically on a complex 
affine variety $X$. Let $G$ be a complex linear algebraic group, and let $E_G$ be a 
$\Gamma$--equivariant principal $G$--bundle on $X$. Given an algebraic group 
homomorphism $\rho \,:\, \Gamma \,\longrightarrow\, G$, the product bundle $X 
\times_{\rho} G$ equipped with diagonal action of $\Gamma$ is the simplest example of a 
$\Gamma$--equivariant principal $G$--bundle. Here we are interested in the 
following question:

\begin{question}\label{q1.1}
Is $E_G$ equivariantly isomorphic to the above product (trivial) principal $G$--bundle
$X \times_{\rho} G$ for some homomorphism $\rho$\,? 
\end{question}

When the structure group $G \,=\,{\rm GL}(n,\mathbb C)$ with $\Gamma$ being the trivial 
group, and $X$ is an affine space, Question \ref{q1.1} was first studied by Serre, and 
he proved that any vector bundle on an affine space is stably trivial \cite{Serre}. 
Later Quillen \cite{Qui} and Suslin \cite{Sus} independently proved that the answer is 
affirmative for the above case of $G \,=\,{\rm GL}(n,\mathbb C)$. This was generalized 
to the case when $X$ is an affine toric variety by Gubeladze \cite{Gub}. When $\Gamma$ 
is trivial, but $G$ is an arbitrary connected reductive algebraic group, and $X$ is an 
affine space, Raghunathan, \cite{Raghu}, showed that the answer is again affirmative.

When $\Gamma$ is arbitrary, it is known that Question \ref{q1.1} has a negative answer 
in general even for vector bundles. This was first shown by Schwarz \cite{Sch}; see 
also the work of Masuda and Petrie \cite{MP1}. However, when $\Gamma$ is a complex torus 
$T$ with $X$ being an affine toric variety under an action of $T$, and $G\,=\,{\rm 
GL}(n,\mathbb C)$, Kaneyama, \cite{Kan}, and Klyachko, \cite{Kly}, gave affirmative answer 
to Question \ref{q1.1}. Moreover, when $\Gamma$ is abelian, $X$ is a linear 
representation space of $\Gamma$, and $G\,= \,{\rm GL}(n, \mathbb{C})$, 
Masuda, Moser and Petrie gave an affirmative answer in \cite{MMP}. Later Masuda resolved
affirmatively the 
case where $\Gamma$ is abelian, $X$ is an affine toric $\Gamma$--variety and $G\,=\, 
{\rm GL}(n, \mathbb{C})$; thus this work of Masuda, \cite{Mas},
generalizes the results of Kaneyama and Klyachko.

In this paper we give an affirmative answer to Question \ref{q1.1} when the following conditions
hold:
\begin{itemize}
\item $\Gamma$ is a reductive complex algebraic group,

\item $X$ is a topologically contractible normal affine variety,

\item $\Gamma$ acts on $X$ with a dense orbit, and

\item $G$ is an arbitrary linear algebraic group.
\end{itemize}
See Theorem \ref{cod2}.

The proof of Theorem \ref{cod2} combines the classically known topological triviality 
of related bundles with an equivariant Oka principle established by Heinzner and 
Kutzschebauch, \cite{HK}, to infer analytic triviality; after that, applying the 
classical fact, mainly due to Borel, that analytic representations of a reductive 
algebraic group are actually algebraic, the algebraic triviality of bundles
in question is inferred. As 
an application we prove that any $T$-equivariant principal $G$--bundle over an affine 
toric variety is $T$--equivariantly trivial. This was proved earlier in \cite{BDP}
under the extra assumption that $X$ is 
a nonsingular affine toric variety. Using Theorem \ref{cod2} and the method of 
\cite{BDP}, we give a Kaneyama--type classification of $T$-equivariant principal 
$G$--bundles (up to isomorphism) over an arbitrary toric variety.

After this article was accepted for publication, Gerald W. Schwarz informed us of a more recent and  stronger version of
 equivariant Oka principle due to Kutzschebauch-L\'arusson-Schwarz \cite{KLS4} that would apply to generalized holomorphic
 principal bundles over a reduced Stein space. See also their related works \cite{KLS1, KLS2, KLS3}.

\section{Main results}

All the objects are defined over the field $\mathbb C$ of complex numbers unless 
mentioned specifically. Let $\Gamma$ be a complex reductive algebraic group with a maximal 
compact subgroup $K$. This means that $\Gamma \,=\, K^{\mathbb C}$, the 
complexification of $K$. Let $G$ be a complex affine algebraic group and $X$ is an affine variety. 
Assume that $\Gamma$ acts on $X$ with a dense 
orbit $O$ of it in $X$. In particular, $\Gamma$ may be an 
algebraic torus and $X$ a toric variety.
 
Let $E_G$ be a $\Gamma$--equivariant algebraic principal $G$--bundle over $X$. This 
means that the action of $\Gamma$ on $X$ lifts to an action on $E_G$ which commutes 
with the action of $G$ on $E_G$ defining the principal $G$--bundle structure of $E_G$, in
other words,
$$ \gamma(eg)\,=\, (\gamma e)g \quad {\rm for \; all} \; \gamma \in \Gamma, \; e \in E_G, \; g\in G\,. $$ 
Given an algebraic group homomorphism $\rho\,:\Gamma \longrightarrow G$, the product bundle 
$$X \times G\,\longrightarrow\, X$$ equipped with the diagonal action of $\Gamma$
naturally becomes a
$\Gamma$--equivariant principal $G$--bundle;
this $\Gamma$--equivariant principal $G$--bundle is denoted by $X \times_{\rho} G$.
We call $X \times_{\rho} G$ a $\Gamma$--equivariant product bundle.

Let $X(\mathbb R)$ denote the underlying topological space for $X$ with real norm topology. Let 
$\Gamma(\mathbb R)$ and $G(\mathbb R)$ denote the Lie groups underlying $\Gamma$ and 
$G$ respectively. Note that the action of $\Gamma$ on $X$ induces an action of 
$\Gamma(\mathbb R)$ on $X(\mathbb R)$. The affine variety $X$ is said to be 
equivariantly contractible as a topological $\Gamma$--space if there exists a 
point $x_0 \,\in\, X(\mathbb R)$ and a $\Gamma(\mathbb R)$--equivariant continuous map
$$F\,:\, X(\mathbb R) \times [0,1]
\,\longrightarrow \,X(\mathbb R)\, ,$$ such that $F(x,1)\,=\, 
x$ and $F(x,0) \,=\,x_0$ for all $x \,\in\, X(\mathbb R)$. Note that $x_0$ is
fixed by the action of $\Gamma(\mathbb R)$. Since $\Gamma$ is reductive, the good quotient $X/\Gamma$ exists. 
As $X$ has a dense orbit of $\Gamma$, $X/\Gamma$ is a point. Since fixed points inject into the good quotient,
$\Gamma$ has a unique fixed point in $X$. It follows that that $x_0$ is the unique $\Gamma(\mathbb R)$--fixed point in $X(\mathbb R)$. 

\begin{theorem}\label{cod2}  Let $X$ be an irreducible normal affine variety over $\mathbb{C}$.
Suppose $X$ is equivariantly contractible as a topological $\Gamma$--space, 
and $x_0 \,\in\, X$ is the fixed point for the action of $\Gamma$. Assume $X$ has a
dense orbit $O$.  
Then there exists a $\Gamma$--equivariant isomorphism $E_G \,\cong\, X \times_{\rho} E_G(x_0) $ of
algebraic principal $G$--bundles.
\end{theorem}

\begin{proof}
Note that $X$ is a reduced Stein space.

First we claim that the existence of an 
algebraic $\Gamma$--equivariant product bundle structure of $E_G$ is equivalent to the 
existence of an analytic $\Gamma$--equivariant product bundle structure.

To prove the above claim, 
assume that there exists a $\Gamma$--equivariant complex analytic isomorphism $\phi\,:\, X 
\times E_G(x_0)\,\longrightarrow\, E_G$. Fix an element $e\,\in\, E_G(x_0)$. The action of 
$\Gamma$ on $E_G(x_0)$ is given by
$$ \gamma (eg) \,=\, e \rho(\gamma) g \quad {\rm for \; all} \; \gamma \in \Gamma, \; g\in G\, ,$$ 
where $$\rho: \Gamma \longrightarrow G$$ is a holomorphic (hence, algebraic) homomorphism
of groups.
Define an analytic section $s$ of $E_G$ by $s(x)\,=\, \phi(x,e)$. Then we have
$$ \gamma s(x) = s(\gamma x) \rho(\gamma) \,.$$ This implies that 
\begin{equation}\label{eqsec}
s(\gamma x) \,=\, \gamma s(x) \rho(\gamma^{-1}) \,.
\end{equation}
Let $x$ be a closed point in the dense orbit $O \,\subset\, X$ of $\Gamma$. Since the action of $\Gamma$ and the homomorphism 
$\rho$ are both algebraic, it follows from \eqref{eqsec} that $s$ is an algebraic section of $E_G$ over $O$. By applying 
Lemma \ref{valuation},  $s$ is regular on $X$.

Secondly, by a corollary of the Homotopy Theorem, \cite[p.~341]{HK}, the question of 
existence of an analytic $\Gamma$--equivariant product bundle structure reduces to the question 
of existence of a topological $K$--equivariant product bundle structure.

Finally, $E_G$ is a Cartan $(K \times G(\mathbb R))$--space in the sense of \cite{Pal}. Therefore, it admits a slice at 
any point $e' \,\in\, E_G$. This ensures the existence of a topological $K$--equivariant trivialization of $E_G$ near any 
$x \,\in\, X$ (cf. Proposition 8.10 of \cite{TD}); see also \cite[Corollary 2.11]{Lash}. Then we may apply the equivariant 
homotopy principle (cf. Theorem 8.15 of \cite{TD}) to obtain a topological $K$--equivariant product bundle structure for 
$E_G$ over $X$.
\end{proof}

\begin{lemma} \label{valuation}
Let $X\,=\,{\rm Spec}(A)$ be a complex affine variety.  Assume that $A$ is a normal domain. Then a rational function $\phi$ on $X$ is regular if it  is continuous in the complex analytic topology.
\end{lemma}

\begin{proof}
Let $K$ be the quotient field of $A$, and let $\phi \,\in\, K$ be the given rational function. Let $\mathfrak p$ be a 
height $1$ prime ideal in $A$. Let Since $A$ is a normal domain, the local ring $A_{\mathfrak p}$ is a discrete valuation 
ring. Let $v_{\mathfrak p}$ be its valuation. Since $\phi$ has a continuous extension on $X$, we have
$v_{\mathfrak p}(\phi) \,\ge \,0$. Hence $\phi \,\in\, A_{\mathfrak p}$ for every height $1$ prime ideal $\mathfrak p$. Since $A$ is normal we have $A 
\,=\,\bigcap_{\text{ht}(\mathfrak p)\,=\,1}A_{\mathfrak p}$, see \cite[Corollary 11.4]{Eis} or \cite[Corollary 5.22]{AM}. Hence $\phi$ is regular on $X$. 
\end{proof}

The above results lead to a classification of equivariant principal bundles over arbitrary toric varieties, generalizing the classification  
in the nonsingular case in \cite{BDP}.  Note that every toric variety is normal (cf. \cite[p.~29]{Ful}). 

\begin{corollary}\label{affinetoric} Let $X$ be an affine toric variety.
Then a torus equivariant algebraic
principal $G$--bundle over $X$ admits an equivariant product bundle structure.
\end{corollary} 
 
\begin{proof}
This follows immediately from Theorem \ref{cod2} if $X$ has a fixed point of the torus action. In 
the general case, the proof is similar to the proof of \cite[Lemma 2.8]{BDP}. Let $T$ denote 
the dense torus of $X$. The idea is to write $X\,=\, Y \times O$ and $T\,=\, H \times K$ where 
$H$, $K$ are sub-tori of $T$ that act on $Y$, $O$ respectively, such that $Y$ is an affine toric 
variety which is $H$--equivariantly contractible with dense a $H$--orbit, and $K$ acts freely, 
transitively on $O$. The variety $X\,=\, Y \times O$ is endowed with the diagonal action of $H\times K$. Let $E_G$ be a 
$T$--equivariant principal $G$-bundle over $X$. By Theorem \ref{cod2}, there exists an 
$H$--equivariant section (often called semi--equivariant in the literature) of $E_G$ over $Y$.  
Then one extends this section to a $T$--equivariant section of $\mc{E}$ over $X\,=\, Y \times O$ by 
using the lift to $E_G$ of the free, transitive action of $K$ on $O$.
\end{proof}

Now let $X$ be a toric variety corresponding to a fan $\Xi$. Let $T$ be the 
dense torus of $X$. For each cone $\sigma \in \Xi$, let $T_{\sigma}$ denote the stabilizer of the $T$--orbit 
corresponding to $\sigma$. For each $\sigma$, fix a projection homomorphism $\pi_{\sigma}: T \longrightarrow 
T_{\sigma}$ which restricts to the identity map on $T_{\sigma}$.

\begin{defn} Let $\Xi^*$ denote the set of maximal cones in $\Xi$. An admissible collection
$\{\rho_{\sigma}, P(\tau,\sigma) \}$
consists of a collection of homomorphisms $$\{\rho_{\sigma}\,:\,T \,\longrightarrow\, G
\,\mid\, \sigma \,\in\, \Xi^* \}$$ and a collection of elements
$\{ P(\tau,\sigma) \in G \mid \tau, \sigma \in \Xi^*\} $ satisfying the
following conditions:
\begin{enumerate} \item $\rho_{\sigma}$ factors through $\pi_{\sigma}: T \longrightarrow
T_{\sigma}$.

 \item For every pair $(\tau,\sigma)$ of maximal
cones, $ \rho_{\tau} (t) P(\tau, \sigma) \rho_{\sigma}^{-1}(t)$ extends to a $G$--valued
regular algebraic function over $X_{\sigma} \bigcap X_{\tau}$.

\item $P(\sigma, \sigma) = 1_G$ for all $\sigma$.

\item For every triple $(\tau, \sigma, \delta)$ of
maximal cones, the cocycle condition
$ P(\tau,\sigma) P(\sigma, \delta)$ $P(\delta, \tau) \,=\, 1_G$ holds.
\end{enumerate}

Two such admissible collections $\{\rho_{\sigma}, P(\tau,\sigma) \}$ and
$\{\rho'_{\sigma}, P'(\tau,\sigma) \}$ are equivalent if the following hold:
\begin{enumerate}
\item[(i)] For every $\sigma$ there exists an element $g_{\sigma}\,\in\, G$ such
that $\rho'_{\sigma} \,=\, g_{\sigma}^{-1} \rho_{\sigma} g_{\sigma}$.

\item[(ii)] For every pair $(\tau,\sigma)$, $P'(\tau,\sigma) \,= \,
g_{\tau}^{-1} P(\tau,\sigma) g_{\sigma}$, where $g_{\sigma}$ and $g_{\tau}$ are as in
(i) above.
\end{enumerate}
\end{defn}

Using Corollary \ref{affinetoric}, the following classification theorem can be proved in a similar manner 
as Theorem 3.2 of \cite{BDP}.

\begin{theorem}\label{class} Let $X$ be a toric variety. Let $G$ an affine
algebraic group. Then the isomorphism classes of algebraic $T$--equivariant
principal $G$--bundles on $X$ are in one-to-one correspondence with equivalence classes of
admissible collections $\{\rho_{\sigma}, P(\tau,\sigma) \}$.
\end{theorem}

The following is then obtained exactly as Corollary 3.4 of \cite{BDP}.

\begin{corollary}\label{nil} If $G$ is a nilpotent group, then an algebraic
$T$--equivariant principal $G$--bundle over a toric
variety  admits an equivariant reduction of structure group to a
maximal torus of $G$. In particular, if $G$ is unipotent then the
bundle is trivial with trivial $T$--action.
\end{corollary}

{\bf Acknowledgement.} The last named-author thanks Peter Heinzner and Ozcan Yazici for useful discussions on equivariant Oka principle.
The authors also thank an anonymous referee for  interesting questions and helpful suggestions that led to several improvements in the manuscript. The first-named author is supported by a J. C. Bose fellowship. The second-named author is supported by a research grant from NBHM. The
last-named author is supported by an SRP grant from METU NCC.


\begin{thebibliography}{BPV84}

\bibitem{AM}  M. F. Atiyah, and I. G.  Macdonald,  Introduction to commutative algebra. Addison-Wesley Publishing Co., Reading, 
Mass.-London-Don Mills, Ont. 1969.

\bibitem{BDP} I. Biswas, A. Dey, and M. Poddar, A classification of equivariant principal bundles over nonsingular 
toric varieties, {\it Internat. J. Math.} {\bf 27} (2016), no. 14, 16 pp.

\bibitem{TD} T. tom Dieck, {\it Transformation groups}, De Gruyter Studies in Mathematics, 8. Walter de Gruyter and Co., Berlin (1987), 312 pp.

\bibitem{Eis} D. Eisenbud, {\it Commutative algebra with a view toward algebraic geometry}, Graduate Texts in Mathematics {\bf 150}, Springer-Verlag, New York, 1995.

\bibitem{Ful} W. Fulton, {\it Introduction to toric varieties}, Annals of Mathematics Studies,
131, Princeton University Press, Princeton, 1993.


\bibitem{Gub} J. Gubeladze, The Anderson conjecture and a maximal class of monoids over which projective modules are free
 (Russian) {\it Mat. Sb. (N.S.)} {\bf 135}(177) (1988), no. 2, 169--185; translation in Math. USSR-Sb. {\bf 63} (1989), no. 1, 165--180

\bibitem{HK} P. Heinzner and F. Kutzschebauch, An equivariant version of Grauert's
Oka principle, {\it Invent. Math.} {\bf 119} (1995), 317--346.

\bibitem{Kan} T. Kaneyama, On equivariant vector bundles on an almost homogeneous
variety, {\it Nagoya Math. Jour.} {\bf 57} (1975), 65--86.

\bibitem{Kly} A. A. Klyachko, Equivariant bundles over toric varieties,
 (Russian) {\it Izv. Akad. Nauk SSSR Ser. Mat.} {\bf 53} (1989), 1001--1039,
1135; translation in Math. USSR-Izv. {\bf 35} (1990), no. 2, 337--375.

\bibitem{KLS1} F. Kutzschebauch, F.  L\'arusson, and G. W. Schwarz, An Oka principle for equivariant
isomorphisms, {\it J. reine angew. Math.}  {\bf 706} (2015), 193--214.

\bibitem{KLS2} F. Kutzschebauch, F. L\'arusson, and G. W. Schwarz, Homotopy principles for equivariant
isomorphisms, {\it Trans. Amer. Math. Soc.} {\bf 369} (2017), no. 10, 7251--7300.

\bibitem{KLS3} F. Kutzschebauch, F.  L\'arusson, and G. W. Schwarz,  Sufficient conditions for holomorphic linearisation,  {\it Transform. Groups}  {\bf 22} (2017), no. 2,
475--485.

\bibitem{KLS4} F. Kutzschebauch, F. L\'arusson, and G. W. Schwarz, An equivariant parametric Oka principle for bundles of homogeneous spaces, {\it Math. Ann.} {\bf 370} (2018), no. 1-2, 819--839.

\bibitem{Lash}  Lashof, R. K. Equivariant bundles. {\it Illinois J. Math.} 26 1982, no. 2, 257--271. 

\bibitem{Mas} M. Masuda, Equivariant algebraic vector bundles over affine toric varieties,{ \it Transformation group theory }
(Taejon, 1996), 76--84, Korea Adv. Inst. Sci. Tech., Taejon, [1997]. 

\bibitem{MMP} M. Masuda, L. Moser-Jauslin and T. Petrie,  The equivariant Serre problem for abelian groups, {\it Topology} 
{\bf 35} (1996), 329--334.


\bibitem{MP1} M.  Masuda and T. Petrie,  Stably trivial equivariant algebraic vector bundles, {\it J. Amer. Math. Soc. } {\bf 8} (1995), no. 3, 687--714.

\bibitem{Pal} R. S. Palais, On the existence of slices for actions of non-compact Lie groups, {\it Ann. of Math.} {\bf 
73} (1961), 295--323.

\bibitem{Qui} D. Quillen, Projective modules over polynomial rings, {\it Inventiones Mathematicae} {\bf 36} no. 1,  (1976) 167--171. 

\bibitem{Raghu} M. S. Raghunathan,  Principal bundles on affine space, {\it C. P. Ramanujam--a tribute, Tata Inst. Fund. Res. Studies in Math.} {\bf 8} (1978), pp. 187--206. 

\bibitem{Sch} G. W. Schwarz,  Exotic algebraic group actions. {\it C. R. Acad. Sci. Paris SŽr. I Math.} {\bf 309}, no. 2, 1989, 89--94.

\bibitem{Serre} J.-P. Serre, Faisceaux alg\'ebriques coh\'erents, {\it Ann. of Math.} {\bf 61} (1955), 197--278,

\bibitem{Sus} A. A. Suslin, Projective modules over polynomial rings are free, {\it Doklady Akademii Nauk SSSR} (in Russian), {\bf 229}, no. 5 (1976): 1063--1066. Translated in "Projective modules over polynomial rings are free", Soviet Mathematics, {bf 17} no. 4, (1976) 1160--1164.
\end{thebibliography}
\end{document}